\documentclass{amsart}
\usepackage{amscd,amssymb,array,amsmath,epsfig,booktabs}
\usepackage{graphicx,color, xcolor}

\newtheorem{theorem}{Theorem}[section]

\theoremstyle{definition}

\theoremstyle{remark}
\newtheorem*{remark}{Remark}

\numberwithin{equation}{section}

\newcommand{\ackname}{Acknowledgements}

\begin{document}
\title[Galoisian approach to Hill equations]
 {Galoisian approach to complex oscillation theory of some Hill equations}
 
\author{Yik-Man Chiang}
\address{Department of Mathematics, Hong Kong University of Science and
Technology, Clear Water Bay, Kowloon, Hong Kong SAR}
\email{machiang@ust.hk}
\thanks{The first author is partially supported by the Hong Kong Reseach Grant Council (GRF no. 16300814 and 601111. The second author is partially supported by National Natural Science Foundation of China (Grant No. 11871336).}

\author{Guo-fu Yu}
\address{Department of Mathematics,
Shanghai Jiao Tong University, Shanghai, 200240, P. R.\ China}
\email{gfyu@sjtu.edu.cn}


\begin{abstract}
We apply Kovacic's algorithm from differential Galois theory to show that all complex non-oscillatory solutions (finite exponential of convergence of zeros) of certain Hill equations considered by Bank and Laine using Nevanlinna theory must be Liouvillian solutions. That are solutions obtainable by suitable differential field extensions construction. In particular, we have established a full correspondence between solutions of non-oscillatory type  and Liouvillian solutions for a particular Hill equation. Explicit closed-form solutions are obtained via both methods for this Hill equation whose potential is a combination of four exponential functions in the Bank-Laine theory. The differential equation is a periodic form of biconfluent Heun equation. We further show that these Liouvillian solutions exhibit novel single and double orthogonality and a Fredholm integral equation over suitable integration regions in $\mathbf{C}$ that mimic single/double orthogonality for the corresponding Liouvillian solutions of the Lam\'e and Whittaker-Hill equations, discovered by Whittaker and Ince almost a century ago. 
\end{abstract}
\maketitle

\section{Introduction} Differential Galois theory {(\cite{Acosta_Blazquez_2008}, \cite{Acosta_Morales_Weil_2011}, \cite{Morales-Ruiz1999}, \cite{Morales-Ruiz2015},  \cite{Morales_Simo_1994}) has been demonstrated a powerful tool to study, amongst various research areas,  monodromy of linear differential equations. In particular, it can be used to identify \textit{Liouvillian} solutions of the differential equations. These solutions correspond to the associated differential Galois group of the differential equations being solvable are closely related to finding closed-form solutions of the differential equations. A  well-known theorem of Kimura \cite{Kimura1969}, \cite{Morales-Ruiz1999} stated that Gauss hypergeometric equation that admits algebraic solutions or else \textit{Jacobi polynomials} are amongst the Liouvillian solutions of the equation.
 The classification given by Kimura contains the celebrated \textit{Schwarz's list} of when the monodromy group of the hypergeometric equation is \textit{finite}. The study of Liouvillian solutions has been extended to the classical Lam\'e equation
  \begin{equation}
        \label{E:lame-1}
            w^{\prime\prime}(z)+[h-n(n+1)\,k^2\textrm{sn}^2z]\,w(z)=0,
    \end{equation}
where the $\textrm{sn}\, z$ is a \textit{Jacobi elliptic function}. The classical \textit{Lam\'e polynomials} are amongst the Liouvillian solutions. One can find a detailed account of the Lam\'e theory in \cite[Chap. 2]{Morales-Ruiz1999}.  Let $n(r,\,f)$ be the number of zeros of an entire function $f$ in $D(0,\, r):=\{z:\, |z|<r\}$. In this paper, we study the interplay between the Liouvillian solutions to the Hill equation
	\begin{equation}\label{E:BHE-periodic}
			f^{\prime\prime}(z)+\big(-e^{4z}+K_{3}\, e^{3z}+K_{2}\, e^{2z}+K_{1}\, e^{z}+K_0\big)f(z)=0,
		\end{equation}
 where $K_j\in \mathbf{C},\ j=0,\, 1,\, 2,\, 3$ and its
 \textit{non-oscillatory solutions}, i.e., those solutions with $\lambda(f)=\limsup_{r\to\infty}{\log n(r,\, f)}/{\log r} <\infty$ or
 \textit{finite exponent of convergence of zeros}  of Bank-Laine's \textit{complex oscillation theory} that studies solutions of linear second order differential equations with transcendental entire coefficients under the framework of Nevanlinna's value distribution theory \cite{Hayman1975}, \cite{Laine1993}, \cite{Yang1993}.

 Bank and Laine \cite{BL1982}, \cite{BL1983}  showed that when  the Hill equation
\begin{equation}
		\label{E:BL-general}
        f^{\prime\prime}(z)+A(z)f(z)=0,
    \end{equation}
where $A(z)=B(e^z)$,
    \begin{equation}\label{E:periodic-coeff}
        B(\zeta)={K_s}/{\zeta^s}+\cdots+K_0+\cdots+K_\ell \zeta^\ell,
    \end{equation}
and $K_j\, (k=s,\cdots,\, \ell)$ are complex constants, admits a non-trivial entire solution $f$ with  $\lambda(f)<+\infty$ (i.e., zero as a \textit{Borel exceptional} value), then
	\begin{equation}\label{E:BL-representation}
        f(z)=\psi(e^{z/h})\exp \big(dz+P(e^{z/h})\big)
    \end{equation}
where $h=2$ if $\ell$ is odd and $h=1$ otherwise, $\psi(\zeta)$ is a polynomial, $P(\zeta)$ is a Laurent polynomial and $d$ is a constant. Obviously, both the polynomial $\psi$ and $P$ depend on the coefficient $B(\zeta)$. See also \cite{Shimomura2002} for higher order differential equations.

Let $x=e^{z/h}$. Then the equation \eqref{E:BL-general} can be transformed to the equation in $x$:
\begin{equation}
		\label{E:bessel-generalised}
		x^2\Psi^{\prime\prime}(x)+x\Psi^\prime(x)+h^2\Big(\sum_{j=s}^{\ell}K_j x^{jh} \Big)\Psi(x)=0.
	\end{equation}
	
The main purpose of this paper is to show the solutions $\Psi(x)=f(x)$ to \eqref{E:bessel-generalised} must be \textit{Liouvillian solutions} in the sense of \textit{Picard-Vessiot theory} of Kolchin \cite{Kolchin1973} when the solutions $f(z)$ of \eqref{E:BL-general} with zero as the Borel exceptional value. That is, $f$ can be represented in the form \eqref{E:BL-representation}.
In the special case of the equation \eqref{E:BHE-periodic}, which is labelled PBHE, we further show that the converse of the above statement also holds, namely, that each Liouvillian solution $\Psi(x)$ to the differential equation \eqref{E:bessel-generalised} must correspond to a solution $f(z)=\Psi(x)$ of the \eqref{E:BHE-periodic} that
assume the form  \eqref{E:BL-representation}, that is, $\lambda(f)<+\infty$. Indeed, the equivalence problem has been solved for the equation $f^{\prime\prime}(z)+(e^z-K_0)f(z)=0$ and the \textit{Morse equation} in \cite{BLL1986}, \cite{CI2006} and \cite{BL1983} and their differential Galois theory counterparts are dealt with, for examples, in \cite[p. 275]{Acosta_Blazquez_2008}, \cite[p. 351]{Acosta_Morales_Weil_2011} and \cite[p. 355]{Acosta_Morales_Weil_2011} respectively. We next show that if $K_3,\, K_2$ are real and $K_0<0$, then these Liouvillian solutions possess  new single and double orthogonality for the non-oscillatory solutions, as well as a novel Fredholm integral equation that resemble the single and double orthogonality for the Lam\'e polynomials (already recorded in \cite[\S95, \S98]{Heine}) and Fredholm integral equation for the Lam\'e equation discovered by Whittaker \cite{Whitt1915}. We also mention that the orthogonal eigenfunction over $(-\infty,\, \infty)$ satisfy the Sturm-Liouville type boundary condition
	\begin{equation}
		\label{E:boundary}
			\lim_{z\to \pm \infty} f(z)=0.
	\end{equation}
In \cite{Blazquez_Yagasaki_2012} the authors applied differential Galois theory to study Sturm-Liouville problems of Schr\"odinger equations in connection with the boundary condition \eqref{E:boundary}.

Our argument is heavily based on the celebrated algorithm due to Kovacic \cite{Kovacic1986}  specifically to differential equations written in the normalised form
		\begin{equation}
			\label{E:DE-normalised}
			y^{\prime\prime}=ry,
		\end{equation}
where $r\in \mathbf{C}(x)$, where $\mathbf{C}(x)$ denotes the field of rational functions. The algorithm is based on a key theorem of Kovacic which is tailored from Kolchin's theory \cite{Kolchin1973}.
\begin{theorem}[(Kovacic, 1986)]\label{T:classification} There are four cases to distinguish.
		\begin{enumerate}
			\item[(1)] The differential equation \eqref{E:DE-normalised} has a solution of the form $\exp\big(\int \omega\big)$, where $\omega(x)\in \mathbf{C}(x)$;
			\item[(2)] The differential equation \eqref{E:DE-normalised} has a solution of the form $\exp\big(\int \omega\big)$, where $\omega(x)$ is algebraic over $\mathbf{C}(x)$ and  case (1) does not hold;
			\item[(3)] All solutions of the differential equation \eqref{E:DE-normalised} are algebraic over $\mathbf{C}(x)$ and cases (1) and (2) do not hold;
			\item[(4)] The differential equation \eqref{E:DE-normalised} has no Liouvillian solution.
		\end{enumerate}
	\end{theorem}
We  state our first  result for the general Hill equation \eqref{E:BL-general}.
\begin{theorem}\label{T:oscillation-galois} Let $f$ be an entire function solution to the Hill equation \eqref{E:BL-general} with finite exponential of convergence of zero, i.e., $\lambda(f)<+\infty$. Then the equation \eqref
{E:bessel-generalised}
where $h=2$, if both $\ell\ge s\ge 0$ are odd and $h=1$, if both  $\ell\ge s$ are even (see e.g., \cite{Chiang2000}), admits a Liouvillian solution  $\Psi(x)=f(z)$ to \eqref {E:bessel-generalised},
  where $x=e^{z/h}$, that belongs to the case (1) of Theorem \ref{T:classification}.
\end{theorem}

Acosta-Hum\'anez and Bl\'azquez-Sanz applied differential Galois theory in \cite{Acosta_Blazquez_2008} to conclude that both the \textit{Mathieu} equation
	\[
		f^{\prime\prime}(z)+(A+B\cos 2z)f(z)=0
	\]
and the \textit{extended Mathieu equation}
	\[
		f^{\prime\prime}(z)+(A+B\cos 2z+C\sin 2z)f(z)=0
	\]
have no Liouvillian solutions.

In addition to the Lam\'e equation (see \S\ref{S:integral}) and the Mathieu equation, Whittaker appears to be the first one to derive a Fredholm type integral equation of the second type for the Whittaker-Hill equation \cite{Whitt1914} (see also \cite{Magnus_Winkler1979}), namely that periodic solutions of the equation
	\begin{equation}
		\label{E:WH-eqn}
		f^{\prime\prime}(z)+\big[A-(n+1)\ell \cos 2z+\frac18\ell^2\cos^2 4z\big] f(z)=0,
	\end{equation}
are eigen-solutions to the
	\begin{equation}
		\label{E:WH-int}
		f(z)=\lambda\int_0^{2\pi} \cos^n(z-t)\, e^{\frac12\ell(\sin^2z+\sin^2 t)}\, f(t)\, dt\, ,
	\end{equation}
with symmetric kernel. We shall extend Ince's method from \cite{Ince1922} to derive an analogous Fredholm type integral equation of the second type for the Liouvillian solutions for the PBHE \eqref{E:BHE-periodic} with symmetric kernel written in terms of a confluent hypergeometric function. These integral equations possess rich monodromy information of the eigen-solutions concerned that are awaiting for serious study, despite their long history.

This paper is organised as follows. We shall briefly review the Picard-Vessiot theory, the meaning of Liouvillian solutions and Kovacic's algorithm in section \S\ref{S:algorithm} before we can state the main results
 about the equivalence between non-oscillatory solutions of \eqref{E:BHE-periodic} and Liovuiliian solutions of
 \eqref{E:bessel-generalised}  in \S\ref{S:equivalence}. The proofs of Theorems \ref{T:determinant_soln} and \ref{T:oscillation-galois} are given in  \S\ref{S:equivalence} and \S\ref{S:oscillation-galois} respectively. We introduce main results about orthogonality in \S\ref{S:orthogonality} and their proofs are  given in \S\ref{S:simple-orth}  and \S\ref{S:double-orth}.  Finally we discuss a new Fredholm type integral equation for the Liouvillian solutions from Theorem \ref{T:determinant_soln} in \S\ref{S:integral} and its derivation is given in \S\ref{S:integral}. We append part of the Kovacic algorithm that we need to use in the Appendix of this paper.

 \section{The Kovacic algorithm} \label{S:algorithm}
\subsection{Differential Galois theory}

We give a brief description on basic terminology and fundamentals of the differential Galois theory that lead to Kovacic's Theorem \ref{T:classification} in this section.

A \textit{differential field} $\mathbf{F}(\supseteq \mathbf{C}(x))$ is \textit{Liouvillian} if there is a tower of differential extension fields $\mathbf{C}(x)=\mathbf{F}_0\subset \mathbf{F}_2\subset \cdots\subset \mathbf{F}_n=\mathbf{F}$ such that for each $i=1,\, \cdots, n$
	\begin{enumerate}
		\item $\mathbf{F}_k=\mathbf{F}_{k-1}[\alpha]$ were $\alpha^\prime/\alpha\in \mathbf{F}_{k-1}$;
		\item or $\mathbf{F}_k=\mathbf{F}_{k-1}[\alpha]$ where $\alpha^\prime\in \mathbf{F}_{k-1}$;
		\item or $\mathbf{F}_k$ is finite algebraic over $\mathbf{F}_{k-1}$.
	\end{enumerate}

A function is said to be \textit{Liouvillian} if it is contained in some Liouvillian field defined above.

Suppose we are given a linear differential equation  $aY^{\prime\prime}+bY^\prime+cY=0$. Then we can always rewrite it into the so-called normal form: $y^{\prime\prime}=ry$ where $r$ belongs to the same differential field as that of $a,\, b,\, c$, that is in $\mathbf{C}(x)$. Let $y_1,\, y_2$ be a fundamental set of solutions to the DE $y^{\prime\prime}=ry$. Let
		\[
			\mathbf{G}=\mathbf{C}(x)<y_1,\, y_2>=\mathbf{C}(x)(y_1,\, y_2,\, y_1^\prime,\, y_2^\prime)
		\]
 be the extension field of $\mathbf{C}(x)$ generated by $y_1,\, y_2$. Then we define the \textit{(differential) Galois group} of $\mathbf{G}$ over $\mathbf{C}(x)$ to be the set of all \textit{differential automorphisms} $\sigma$ of $\mathbf{G}$  that leaves $\mathbf{C}(x)$ invariant. An automorphism $\sigma$ of $\mathbf{G}$ is differential  if $\sigma(y^\prime)=(\sigma y)^\prime$ for all $\sigma \in \mathbf{G}$.

Without going into the details, the fundamental theorem of differential Galois theory of a linear differential equation is about the isomorphisms between the intermediate fields of Picard-Vessiot extension and the closed subgroups of its differential Galois group. It is fundamental and demonstrated by Kovacic \cite{Kovacic1986} that the representations of these subgroups are \textit{linear algebraic groups}, namely that they are subgroups of $GL(2,\, \mathbf{C})$, in that each matrix element satisfies a polynomial equation. A fundamental theorem of Lie-Kolchin \cite[p. 8]{Morales-Ruiz1999},
is that a connected linear algebraic group is \textit{solvable} if and only if it is conjugate to a triangular group. See also Laine \cite[Chap. 14]{Laine1993}. Then a classification of algebraic subgroups of $SL(2,\, \mathbf{C})$ is that
	\begin{enumerate}
		\item $\mathbf{G}$ is triangulisable.
		\item $\mathbf{G}$ is conjugate to a subgroup of the infinite dihedral group
		and case (1) does not hold.
		\item $\mathbf{G}$ is finite and cases (1) and (2) do not hold.
		\item $\mathbf{G}=SL(2,\, \mathbf{C})$.
	\end{enumerate}

\subsection{Kovacic's algorithm}	
Kovacic tailored the above theorem specifically to differential equations written in the normalised form $y^{\prime\prime} =ry$  already given in Theorem \ref{T:classification}.

After transforming a given linear differential equation to the normalised form \eqref{E:DE-normalised}, Kovacic's algorithm gives a \textit{necessary condition} that tests the plausibility of the equation \eqref{E:DE-normalised} having a Liouvillian solution. We first write $r=s/t$ with $s,\, t\in\mathbf{C}[x]$, relatively prime. Then it is clear that the finite poles of $r$ locate at the zeros of $t$. We understand the order of $r$ at infinity in the usual sense to be the order of pole of $r(1/x)$ at $x=0$. That is, the integer $\deg t-\deg s$.
 		
\begin{theorem}[Kovacic \cite{Kovacic1986}]\label{T:rational} The following conditions are necessary for the respective cases mentioned in Theorem \ref{T:classification} to hold.
	\begin{enumerate}
		\item Every pole of $r$ must have even order or else have order 1. The order of $r$ at $\infty$ must be even or else be greater than 2.
		\item $r$ must have at least one pole that either has odd order greater than 2 or else has order 2.
		\item The order of a pole of $r$ cannot exceed 2 and the order of $r$ at $\infty$ must be at least 2. If the partial fraction expansion of $r$ is
			\begin{equation}
				r=\sum_i \frac{\alpha_i}{(x-c_i)^2}+\sum_j\frac{\beta_j}{x-d_j},
			\end{equation}
		then $\sqrt{1+4\alpha_i}\in\mathbf{Q}$, for each $i$, $\sum_j\beta_j=0$, and if
			\begin{equation}
				\gamma=\sum_i\alpha_i+\sum_j\beta_jd_j,
			\end{equation}
		then $\sqrt{1+4\gamma}\in\mathbf{Q}$.
	\end{enumerate}
\end{theorem}

One observes from Theorem \ref{T:classification} that the most general solution has the form $y=P\exp\big(\int \omega\big)=\exp\big(\int P^\prime/P+\omega\big)$, where $P\in \mathbf{C}[x]$ and $\omega\in\mathbf{C}(x)$ so that the logarithmic derivative of $y$ satisfies the Riccati equation
	\begin{equation}
		\frac{P^{\prime\prime}}{P}+2\omega\, \frac{P^\prime}{P}+(\omega^\prime+\omega^2)=r
	\end{equation}
and the task would then be to determine the \textit{main part} of $\omega$ using the above Theorem \ref{T:rational} by constructing a rational function denoted by $[\sqrt{r}]_c$ for various poles $c$ of $r$. The remaining task would be to determine the polynomial component $P$ in the differential equation
	\begin{equation}
		\label{E:Riccati-linearlized}
		P^{\prime\prime}+2\omega P^\prime +(\omega^\prime+\omega^2-r)P=0,
	\end{equation}
if any. An affirmative answer would imply a Liouvillian solution. Since the full Kovacic's algorithm is lengthy, we refer readers to the appendix of this paper.

\section{Complex non-oscillatory and Liouvillian solutions}\label{S:equivalence}
It is easy to check that the transformation $\Psi(x)=f(z),\ x=e^{z}$ (i.e., $h=1$ in \eqref{E:BL-representation}), transforms the equation \eqref{E:BHE-periodic} to the equation
	\begin{equation}
		\label{E:bessel-generalised-BHE}
		x^2\Psi^{\prime\prime}(x)+x\Psi^\prime(x)+\big(-x^4+K_3x^3+K_2x^2+K_1x+K_0\big)\Psi(x)=0,
	\end{equation}
which we call a \textit{generalised Bessel's equation}. It is easy to see from the format of the equation that \eqref{E:bessel-generalised-BHE} has a regular singularity at the origin and an irregular singularity at $\infty$ which is of higher rank than that of the Bessel equation. The well-known formula
	\begin{equation}
		\label{E:transform}
		\Psi=y\cdot \exp\Big(-\frac12\int \frac1x\Big)=x^{-1/2}\cdot y
	\end{equation}
transforms the equation \eqref{E:bessel-generalised-BHE} to a standard normal form
	\begin{equation}
		\label{E:BHE-normal-BL}
			y^{\prime\prime}=\Big(x^2-K_3x-K_2-\frac{K_1}{x}-\frac{1/4+K_0}{x^2} \Big)y
	\end{equation}
in order to apply Kovacic's algorithm to identify Liouvillian solutions to the normal form of generalised Bessel equation \eqref{E:BHE-normal-BL}, if any. This would lead to a criterion of entire solutions of the Hill equation \eqref{E:BHE-periodic} with finite exponent of convergence of zeros $\lambda(f)<+\infty$.
The origin is a regular singularity of equation \eqref{E:BHE-normal-BL}, but the $\infty$ is an irregular singularity. We mentioned that the authors of \cite{FRRJT_2010} applied differential Galois theory to study the Stoke phenomenon of \textit{prolate spheroidal wave equation}. It would be interesting to see if their approach could also be applied to this equation we consider in this paper.

\begin{theorem}
    \label{T:determinant_soln}
        Let $K_3,\, K_2,\, K_1$ and $K_0$ be arbitrary constants in $\mathbf{C}$. Then the Hill equation \eqref{E:BHE-periodic}, when written in the operator form,
			\begin{equation}
				\label{E:BHE-operator}
					{-H}f(z):=-\Big(e^{-z}\frac{d^2}{dz^2}-e^{3z}+K_3e^{2z}+K_2e^z+K_0e^{-z}\Big)f(z)
=K_1f(z),
			\end{equation}
admits an entire solution with $\lambda(f)<+\infty$ if and only if the normalised generalised Bessel equation \eqref{E:BHE-normal-BL} admits a Liouvillian solution in the case \emph{(1)} of Kovacic's Theorem \ref{T:classification} above. Moreover, when this happens, there exists an non-negative integer $n$ such that the following equation
		\begin{equation}
        \label{E:condition_I}
           \frac{K_3^2}{4}+K_2+ 2\epsilon_0\epsilon_\infty \sqrt{-K_0}=-2\epsilon_\infty (n+1),
    \end{equation}
holds amongst the $K_3,\, K_2$ and $K_0$, where $\epsilon_\infty^2=\epsilon_0^2=1$. Furthermore,  there are either
        \begin{enumerate}
            \item $n+1$, possibly repeated, choices of $K_1$;\\
  \noindent          or
            \item precisely $n+1$ distinct real values of $K_1$, labelled by $(K_1)_\nu\ (\nu=0,\, \cdots, n)$, if, in addition, that $K_3,\, K_2,\, K_0$ are assumed to be real, $K_0<0$, $\epsilon_\infty=-1$ and such that
                \begin{equation}
                    \label{E:distinct}
                      1+2\epsilon_0\, \sqrt{-K_0}>0
                \end{equation}
                holds.
        \end{enumerate}
The $(K_1)_\nu\ (\nu=0,\, \cdots, n)$ consist of the roots of the $(n+1)\times (n+1)$ determinant $D_{n+1}(K_1)=0$ where $D_{n+1}(K_1)$ is  tridiagonal and it is equals to
      \begin{equation}
        \label{E:Det}
    \left|
    \begin{array}{lllll}
        k_1 & 1 & & &  \\
        (1+2\epsilon_0\sqrt{-K_0})k_2 & k_1-\epsilon_{\infty}K_3 & 1& &  \\
        &
            \begin{split}
                &2(2+2\epsilon_0\sqrt{-K_0})\\
                &\ \times (k_2+2\epsilon_\infty)
            \end{split}
        &k_1-\epsilon_\infty2K_3 & 1 & \\
        &\ddots & \ddots & \ddots & \\
        &\ddots & \ddots & \ddots &\\
        & &
            \begin{split}
                    &(n-1) \\
					&\ \times(n-1+2\epsilon_0\sqrt {-K_0}) \\
                    &\ \times (k_2+2\epsilon_{\infty}(n-2))
            \end{split}
        &  k_1-\epsilon_\infty(n-1)K_3 &  1 \\
        & & &
            \begin{split}
                &n(n+2\epsilon_0\sqrt{-K_0})\\
                &\ \times (k_2+2\epsilon_{\infty}(n-1))
            \end{split}
        & k_1-\epsilon_{\infty}nK_3
    \end{array}
    \right|,
 \end{equation}
 and $k_1$, $k_2$ are given by
\begin{align}\label{E:k_s}
&k_1=K_1-\frac{\epsilon_{\infty}}{2}(1+2\epsilon_0\sqrt{-K_0})K_3,\\
& k_2=2\epsilon_{\infty}(1+\epsilon_0\sqrt{-K_0})+K_2+\frac{1}{4}K_3^2.
\end{align}

Moreover, we can write the solution $f$ in the form
\begin{equation}
        \label{E:BHE-periodic-soln}
            \mathrm{BH}_{n,\,\nu}(z)= Y_{n,\,\nu}(e^{-z})\, \exp\bigg[\frac{\epsilon_\infty}{2}\big(e^{2z}-K_3\, e^z\big)+(n+\epsilon_0\sqrt{-K_0})\,z\bigg],
    \end{equation}
where  $Y_{n,\,\nu}(x)$ is a polynomial of degree $n$, and $\nu=0,\, 1,\, \cdots, n$ respectively for the different roots of $D_{n+1}(K_1)=0$. Clearly, each $y(x)= \mathrm{BH}_{n,\,\nu}(z)\ (\nu=0,\cdots, n$) is a Liovuillian solution to the differential equation \eqref{E:BHE-normal-BL}.
\end{theorem}
 \begin{remark}  We remark that if $\epsilon_0=1$,  so that $1+2\epsilon_0\sqrt{-K_0}>1>0$ holds trivially. If, however,  $\epsilon_0=-1$, then $1-2\sqrt{-K_0}>0$, so that $0<\sqrt{-K_0}<1/2$.
 \end{remark}

\section{Proof of Theorem \ref{T:oscillation-galois} }\label{S:oscillation-galois}
\begin{proof} We suppose that $f$ is a solution of \eqref{E:BHE-periodic} with $\lambda(f)<+\infty$. Then Bank and Laine's result \cite{BL1983} asserts that $f$ is given by \eqref{E:BL-representation}, where $f(z)=\Psi(x)$, $x=e^{z/h}$. Hence
	\begin{equation}
		\label{E:connection}
			\Psi(x) = {x^{hd}}\psi(x)\exp\big(P(x)\big)=\exp\Big(\int \omega\Big)
	\end{equation}
where
	\begin{equation}
		\omega= P^\prime+\frac{\psi^\prime}{\psi} +\frac{hd}{x},
	\end{equation}
which belongs to $\mathbf{C}(x)$. According to the transformation  \eqref{E:transform}, we have
	\begin{equation}
		y=\exp\Big(\int \tilde{\omega}\Big)
	\end{equation}
where
	\begin{equation}
		\tilde{\omega}= P^\prime+\frac{\psi^\prime}{\psi} +\frac{2hd+1}{2x}
		\in \mathbf{C}(x)
	\end{equation}
matches precisely the case 1 of Theorem \ref{T:classification}. Hence the equation
	\[
		x^2 y^{\prime\prime}=h^2\big(-\sum_{j=s}^{\ell}K_j x^{jh} -1/4\big)y
	\]
and \eqref{E:bessel-generalised} possess Liouvillian solution.
\end{proof}

 \section{Proof of Theorem \ref{T:determinant_soln}}\label{S:determinant_soln}

\begin{proof}
Suppose that if a solution $f$ of \eqref{E:BHE-periodic} has $\lambda(f)<+\infty$. Bank and Laine's result asserts that there exist complex constants $d,\, d_j$ and a polynomial $P_n(\zeta)$ of degree $n$ such that
\begin{align}
f(z)=P_n(e^z)\exp\Big(d_1\,e^z+d_2\, e^{2z}+d\,z \Big),\label{bank-laine}
\end{align}
with
\[
P_n(\zeta)=\sum_{k=0}^n q_k\zeta^k,\quad q_n\neq 0.
\]
Substitution of \eqref{bank-laine} into \eqref{E:BHE-periodic} results in
\begin{align}
4d_2^2-1=0,\quad 4d_1d_2+K_3=0,\quad d^2+K_0=0.
\end{align}
Then we introduce notation $\epsilon_{\infty}=\pm 1$, $\epsilon_0=\pm 1$ and write $d_2={\epsilon_{\infty}}/{2}, \, d_1=-{\epsilon_\infty K_3}/{2},\, d=\epsilon_0\sqrt{-K_0}$.
Thus we obtain the solution \eqref{E:BHE-periodic-soln}.

It follows from Theorem \ref{T:oscillation-galois} that if a solution $f$ of \eqref{E:BHE-periodic} has $\lambda(f)<+\infty$, then the function $y=x^{1/2}\Psi(x)$, where $\Psi(x)=f(z)$, $x=e^z$ is a Liouvillian solution to the differential equation \eqref{E:BHE-normal-BL}. Thus, we apply Kovacic's algorithm (case 1) to the equation \eqref{E:BHE-normal}.  It turns out that we can identify the \eqref{E:BHE-normal-BL} with the \textit{biconfluent Heun equation} (BHE) \cite{Ronv1995}, written in the following normalised form,
	\begin{equation}
		\label{E:BHE-normal}
			y^{\prime\prime}=\Big(x^2+\beta x+\frac{\beta^2}{4}-\gamma+\frac{\delta}{2 x}+\frac{\alpha^2-1}{4x^2}\Big) y,
	\end{equation}
where $\alpha,\, \beta,\, \gamma,\, \delta$ are some constants. Equation
	\begin{equation}
		\label{E:BHE}
			xy^{\prime\prime}+(1+\alpha-\beta\,x-2x^2)y^{\prime}
+\Big((\gamma-\alpha-2)x-\frac{1}{2}(\delta+(1+\alpha)\beta) \Big)y=0,
	\end{equation}
is called the canonical form of BHE.
Duval and Loday-Richard \cite{Duval_Loday1992} have already applied Kovacic's algorithm to the equation \eqref{E:BHE-normal} to obtain criteria for Liouvillian solutions. In addition to the fact that they did not give details of their computation in \cite[Prop. 13]{Duval_Loday1992}, see also \cite{Morales-Ruiz2015}, it would be difficult for us to elaborate on the dependence of the different set of coefficients, namely $K_3,\cdots, K_0$ from the \eqref{E:BHE-periodic}. In particular, we shall derive four different sets of $n+1$ solutions of the from \eqref{E:BHE-periodic-soln} as emphasised in the remark after the theorem, so we judge it appropriate in working out details of the algorithm here for the sake of completeness, especially given our very different motivation, context and notation.

Kovacic's algorithm consists of identifying the set of poles  $\{c\}$ (both finite and $\infty$) of $r$, constructing a suitably defined rational function denoted by $[\sqrt{r}]_c$ based on the Laurent expansion of $r$ at the pole $c$, and the corresponding tuple $(\alpha_c^+,\, \alpha_c^-)$ of complex numbers. In the case of \eqref{E:BHE-normal-BL} under consideration, our $r$ can only have poles at $0$ and $\infty$.

According to Kovacic's algorithm (see the Appendix), there are three cases to consider:
	\begin{enumerate}
		\item $K_0=-\frac{1}{4}, K_1=0$,
		\item $K_0=-\frac{1}{4}, K_1\neq 0$,
		\item $K_0\neq -\frac{1}{4}$.
	\end{enumerate}
We use the notation $\mathbf{Z}$ and $\mathbf{N}$ to denote all integers and all positive integers respectively. The generalised Bessel equation \eqref{E:BHE-normal-BL} has Liouvillian solutions if and only if it falls into the Case (1) of Kovacic algorithm and we claim that one of
the following conditions is fulfilled.
\begin{enumerate}
\item $K_0=-\frac{1}{4}, K_1=0$ and $-\epsilon_{\infty}(K_2+\frac{K_3^2}{4})\in 2\mathbf{Z}+1$.  The monic polynomial $P_n(x)$ of degree $n$ in the Liouvillian solution satisfies
\begin{align}\label{E:case1-1}
P_n''+2\epsilon_{\infty}(x-\frac{K_3}{2})P_n'+(\epsilon_{\infty}+K_2+\frac{1}{4}K_3^2)P_n=0.
\end{align}

\item $K_0=-\frac{1}{4},\, K_1\neq 0$ and $-\epsilon_\infty(K_2+\frac{1}{4}K_3^2)=2n+3\in 2\mathbf{N}+1$. Then $P_n(x)$ in the Liouvillian solution satisfies
\begin{align}
	\label{E:case1-2}
& xP_n''(x)+(2 -\epsilon_\infty K_3x+2\epsilon_\infty x^2)P_n'\notag\\
&\qquad \quad+\Big[K_1-\epsilon_\infty K_3+(3\epsilon_\infty +K_2+\frac{1}{4}K_3^2)x \Big]P_n=0.
\end{align}

\item $K_0\neq -\frac{1}{4}$ and  $\epsilon_0,\epsilon_{\infty} \in \{\pm 1\}$ such that $-\epsilon_{\infty}(K_2+\frac{1}{4}K_3^2)-\epsilon_0\sqrt{-4K_0}=2(n+1) \in 2\mathbf{N}$.  $P_n(x)$ in the Liouvillian solution satisfies
\begin{align}\label{E:case1-3}
&xP_n''(x)+(1+2\epsilon_0\sqrt{-K_0}-\epsilon_{\infty} K_3x+2\epsilon_{\infty} x^2)P_n'\nonumber\\
&\quad+\Big[K_1-\frac{1}{2}\epsilon_{\infty}(1+2\epsilon_0\sqrt{-K_0})K_3
+(2\epsilon_{\infty}+2\epsilon_0\epsilon_{\infty}\sqrt{-K_0} +K_2+\frac{1}{4}K_3^2)x \Big]P_n=0.
\end{align}
\end{enumerate}
The case \eqref{E:case1-1} reduces essentially to a Hermite equation whose structure is so degenerate which already falls outside the scope of being a genuine BHE case. The case \eqref{E:case1-2} is a special case of \eqref{E:case1-3} and so we omit the details. In the following we consider the third case when $K_0\neq -\frac{1}{4}$ to illustrate the Kovacic algorithm.
In this situation, we have
	\[
		r(x)=\frac{-1/4-K_0}{x^2}+\frac{-K_1}{x}-K_2-K_3x+x^2.
	\]
We first need to identify the singularities of $r$, and they are $\{0,\, \infty\}$. So the set $\Gamma=\Gamma^\prime\cup\{\infty\}=\{0\}\cup\{\infty\}$ and the only finite pole is $c=0$. Then it is obvious that the order of $r(x)$ at $x=0$ is two, i.e., $o(r_0)=2$. According to Case 1, Step 1 in the Appendix, we deduce that the rational function $[\sqrt{r}]_0\equiv 0$, $r=\cdots+b x^{-2}$ with $b=-1/4-K_0$ so that the tuple $(\alpha_0^+,\, \alpha_0^-)$ of numbers are given by the expression
	\begin{equation}
		\alpha_0^{\epsilon_0}=\frac{1+\epsilon_0\sqrt{1+4b}}{2}=\frac{1+\epsilon_0\sqrt{-4K_0}}{2}
=\frac12+\epsilon_0\sqrt{-K_0},\quad \epsilon_0=\pm 1.
	\end{equation}
	Similarly, it is easy to spot that the only other pole being at $\infty$ also has order $-2$, i.e., $o(r_\infty)=2-4=-2$. Then the Case (1), Step  1 of Kovacic's algorithm in the Appendix again asserts that we have  $[\sqrt{r}]_\infty\equiv x-\frac{K_3}{2}$, and that the tuple $(\alpha_\infty^+,\, \alpha_\infty^-)$ of numbers are given by
\begin{align}
\alpha_\infty^{\epsilon_\infty}=\frac{1}{2}\Big[-\epsilon_\infty(K_2+\frac{K_3^2}{4})-1\Big],
\quad \epsilon_{\infty}=\pm 1.
\end{align}	
	
We continue the algorithm to the next step, the Step 2, where we define the set of positive integers $n$:
\begin{align*}
D=\{ n\in \mathbf{N}: n=\alpha_{\infty}^{\epsilon_{\infty}}-\sum_{c\in \Gamma'}\alpha_c^{\epsilon_c},\ \forall(\epsilon_p)_{p\in \Gamma} \}.
\end{align*}
and we can derive
	\begin{equation}
		\label{E:quantization}
		 n=\alpha_\infty^{\epsilon_\infty}-\alpha_0^{\epsilon_0}=-1+\frac{1}{2}\Big[-\epsilon_\infty\Big(K_2+\frac{K_3^2}{4}\Big)-\epsilon_0\sqrt{-4K_0}\Big],
	\end{equation}
holds for some non-negative integer $n$ according to the assumption \eqref{E:condition_I}. Hence  $1=\#(D)>0$. We now construct the rational function $\omega$ according to the Step 2:
	\begin{equation}
 \omega=\epsilon_\infty\cdot[\sqrt{r}]_\infty+\epsilon_0\cdot[\sqrt{r}]_0+\frac{\alpha_0^{\epsilon_0}}{x}=\epsilon_\infty(x-\frac{K_3}{2})
+\frac{\alpha_0^{\epsilon_0}}{x},
	\end{equation}
so that	
	\begin{equation}
		\label{E:approx-root}
		\omega'+\omega^2-r=2\epsilon_\infty  +K_2+\frac{K_3^2}{4}+\epsilon_\infty\epsilon_0\sqrt{-4K_0}+\frac{K_1-\epsilon_\infty\alpha_0^{\epsilon_0}K_3}{x},
	\end{equation}
with $\epsilon_{\infty}=\pm 1$. Substituting the equation \eqref{E:quantization} or equivalently the equation \eqref{E:condition_I} into \eqref{E:approx-root} simplifying it to
	\begin{equation}
		\label{E:approx-root-2}
		\omega'+\omega^2-r=-2\epsilon_\infty n+\frac{K_1-\epsilon_\infty\alpha_0^{\epsilon_0}K_3}{x}.
	\end{equation}
It remains to solve for the differential equation \eqref{E:Riccati-linearlized} for a polynomial $P_n$
as stated in the Step 3 in Case (1) of Kovacic's algorithm in the Appendix. That is,  we need to solve for the differential equation
\begin{align}\label{E:case1-3*}
xP_n''(x) +(1+2\epsilon_0\sqrt{-K_0} &-\epsilon_{\infty} K_3\,x+2\epsilon_{\infty} x^2)P_n'\nonumber\\
&\quad+\big(-2\epsilon_\infty n x+{K_1-\epsilon_\infty\alpha_0^{\epsilon_0}K_3}\big)P_n=0
\end{align}
which is a simplified form of the equation \eqref{E:case1-3}. Substitute
\begin{align*}
P_n=\sum_{m=0}^n \frac{A_m}{(1+2\epsilon_0\sqrt{-K_0})_m}\frac{x^m}{m!},\quad (\mu)_n=\frac{\Gamma(\mu+n)}{\Gamma(\mu)}, \quad n \geq 0,
\end{align*}
into  \eqref{E:case1-3}, we obtain the three-term recursion relation
\begin{align} \label{E:three-term}
& A_0=1, \quad A_1=-k_1\notag
\\
& A_{m+2}+(k_1-\epsilon_{\infty}(m+1)K_3)A_{m+1}\notag
\\
&\qquad+(m+1)(m+1+2\epsilon_0\sqrt{-K_0})(k_2+2\epsilon_\infty m)A_m=0,
\end{align}
with $k_1$ and $k_2$ given by \eqref{E:k_s}
which results in the determinant \eqref{E:Det} being zero.   That is, there are $n+1$, possibly repeated, roots of $K_1$. Let us first suppose the inequality \eqref{E:distinct} holds when $\epsilon_\infty=-1$. Hence $1+2\epsilon_0\sqrt{-K_1}>0$. However, this expression from equation \eqref{E:case1-3*}  corresponds to the parameter $1+\alpha>0$ of equation \eqref{E:BHE}. It follows from a theorem of Rovder \cite[Theorem 1]{Rovder1974} that the determinant associates with \eqref{E:three-term} admits $n+1$ distinct real values of $K_1-\epsilon_\infty\alpha_0^{\epsilon_0}K_3$ and hence $K_1$. It follows that the equation \eqref{E:case1-3*} has $n+1$ distinct polynomial solutions.

Thus the Liouvillian solution to \eqref{E:BHE-normal-BL} is expressed by
\begin{align}
y(x)=P_n(x)\exp\Big(\int \omega(x)dx\Big)=P_n(x)\, x^{\alpha_0^{\epsilon_0}}\,\exp\Big[\frac{\epsilon_{\infty}}{2}\, (x^2-K_3\,x)\Big].
\end{align}
From the transformation \eqref{E:transform}, we obtain the solution to the Hill equation \eqref{E:BHE-periodic}
\begin{align}
f(z)=P_n(e^z)\exp\Big(\frac{\epsilon_\infty}{2}\,e^{2z}-\frac{\epsilon_\infty}{2}K_3\, e^z+\epsilon_0\sqrt{-K_0}z \Big).
\end{align}
The cases $(1)$ and $(2)$  can be considered similarly.

\end{proof}

 \section{Orthogonal solutions of Hill equation}\label{S:orthogonality}
We next show for the first time that there are novel orthogonality amongst the non-oscillatory/Liouvillian  solutions $\textrm{BH}_{n,\,\nu}(z)$  (\ref{E:BHE-periodic-soln}) for the periodic BHE (\ref{E:BHE-periodic}) from the Bank-Laine theory. These are analogous results for Lam\'e equation which we now review.
\medskip

  It is known that the classical Lam\'e polynomials \cite[\S9.2--9.3]{Arscott1964} are eigen-solutions to the Lam\'e equation
    \begin{equation}
        \label{E:lame-2}
            w^{\prime\prime}(z)+[h-n(n+1)\,k^2\textrm{sn}^2z]\,w(z)=0,
    \end{equation}
where $n$ is chosen to be a non-negative integer, then there are $n/2$ or $(n+1)/2$ distinct choices of $h$, depending on when $n$ is even and odd respectively, thus leading to $2n+1$ distinct choices of \textit{Lam\'e polynomials}, and each of which assumes the form
	\begin{equation}
		\label{E:Lame-poly}
			\mathrm{sn\,}^rz \, \mathrm{cn\,}^sz\,  \mathrm{dn\,}^t z\  F_p(\mathrm{sn\,}^2 z)
	\end{equation}
where $r,\, s,\, t$ can take integer values either $0$ or $1$, subject to the constraint  $r+s+t+2p=n$. This leads to eight different combinations of $(r,\, s,\, t)$, called \textit{species} \cite[p. 201]{Arscott1964}, and a total of $2n+1$ Lam\'e polynomials polynomials. The coefficients of polynomial $F_p$ in (\ref{E:Lame-poly}) satisfies a three-term recursion. It is known from  19th century texts that Lam\'e polynomials of the same type satisfy orthogonality relation given by, for example,
    \begin{equation}
        \label{E:lame-orth}
            \int_{-2K}^{2K} E_n^{m_1}(z) \, E_n^{m_2}(z)\, dz=0
    \end{equation}
whenever $m_1\not= m_2$ ($0\le m_1,\,m_2\le n$), where the \textit{real period} and \textit{imaginary period} of $\textrm{sn\,}z$ are denoted by $4K$ and $2iK^\prime$ respectively (e.g. \cite[p. 370]{Heine}, \cite[p. 466]{Hobson}; see also \cite[\S9.4]{Arscott1964}). That is, two Lam\'e polynomials of the same spice must have the same degree $n$, and their orthogonality is over the other parameter $m\ (0\le m\le n)$, which differs from conventional orthogonality we encounter from other well-known examples.
On the other hand, we have
	\[
		 \int_{-2K}^{2K} (E_n^{m}(z))^2 \, dz\not= 0
	\]
whenever $m_1=m=m_2$ \cite[p. 206]{Arscott1964}.
\medskip

We next show that the set of parameters $(\epsilon_0,\, \epsilon_\infty)=(1,\, -1)$ associates to
 $n+1$ eigen-solutions $\mathrm{BH}_{n,\,\nu}(z)$  defined in (\ref{E:BHE-periodic-soln}) exhibit similar orthogonality as the (\ref{E:lame-orth}). That is, any two functions $ \mathrm{BH}_{n,\,\nu}(z)$ of the same degree $n$ out of the $n+1$  solutions defined by (\ref{E:BHE-periodic-soln}) are orthogonal with respect to the weight $e^z$ over $(-\infty,\, +\infty)$, and \textit{different parameters} $\nu\ (0\le \nu\le n)$ (instead of different degrees $n$).
\medskip

Moreover, these eigen-solutions $\mathrm{BH}_{n,\,\nu}(z)$ satisfy the \eqref{E:boundary}, i.
e.,
	\begin{equation}
		\label{E:boundary_2}
			\lim_{z\to \pm \infty} \mathrm{BH}_{n,\,\nu}(z) =0
	\end{equation}
 which serves as a boundary condition of a Sturm-Liouville problem for time-independent linear Schr\"odinger equation including our equation \eqref{E:BHE-eigenstate-II} below studied by Bl\'azquez-Sanz and Yagasaki \cite{Blazquez_Yagasaki_2012}. There the authors showed that the corresponding eigen-solutions correspond to discrete eigenvalues are Liouvillian.  We note that we require the $\mathrm{BH}_{n,\,\nu}(z)$ to satisfy  the boundary condition \eqref{E:boundary_2} for the orthogonality result below. The results in \cite{Blazquez_Yagasaki_2012} provide  another connection of differential Galois theory with non-oscillatory solutions of \eqref{E:BHE-eigenstate-II} below.

\begin{theorem}\label{T:simple-orth} Let $n$ be a non-negative integer. Suppose that the coefficients $K_3,\, K_2$ and $K_0<0$ of the Hill operator \eqref{E:BHE-operator} are all real, satisfy the relation
		\begin{equation}
        \label{E:condition_III}
           \frac{K_3^2}{4}+K_2-2\sqrt{-K_0}=2(n+1),
    \end{equation}
for each $n$. Then the
 equation
		\begin{equation}\label{E:BHE-eigenstate-II}
       -H\,f(z):=-\Big(e^{-z}\frac{d^2}{dz^2}-e^{3z}+K_3e^{2z}+K_2e^z+K_0e^{-z}\Big)f(z)=\big(K_{1}\big)_{n,\, \nu}\, f,\qquad 0\le\nu\le n,
    \end{equation}
admits $n+1$ linearly independent solutions $\mathrm{BH}_{n,\, \nu}(z),\ (0\le\nu\le n)$ of the form
	\begin{equation}
        \label{E:BHE-periodic-soln-2}
            \mathrm{BH}_{n,\,\nu}(z)= e^{nz}Y_{n,\,\nu}(e^{-z})\, \exp\bigg[-\frac{1}{2}\big(e^{2z}-
            K_3e^z\big)\,+\sqrt{-K_0}\,z\bigg],
    \end{equation}
to \eqref{E:BHE-eigenstate-II} given in \eqref{E:BHE-periodic-soln} corresponding to the $n+1$ distinct real values of $\big(K_{1}\big)_{n,\, \nu}\ (0\le \nu\le n)$ satisfy the orthogonality, namely
    \begin{equation}
		\label{E:simple-orth-2}
            \int_{-\infty}^{\infty} \mathrm{BH}_{n,\, \mu}(x)\ \mathrm{BH}_{n,\, \nu}(x) \,e^x\, dx=0
    \end{equation}
whenever $\mu\not=\nu$, and
	\begin{equation}
		\label{E:single-norm}
			 \int_{-\infty}^{\infty} (\mathrm{BH}_{n,\, \mu}(x))^2 \,e^x\, dx\not=0
    \end{equation}
whenever $\mu=\nu$.
\end{theorem}
\medskip

We further observe that one can define the inner product by integrating the Lam\'e polynomials over a \textit{complex period} $[K-2iK^\prime,\, K+2iK^\prime]$ instead of the real period $[-2K,\, 2K]$ \cite[\S9.4]{Arscott1964}
considered above. That is,
	 \begin{equation}
        \label{E:lame-orth-2}
            \int_{K-2iK^\prime}^{K+2iK^\prime} E_n^{m_1}(z) \, E_n^{m_2}(z)\, dz=0
    \end{equation}
whenever $m_1\not= m_2$.
\bigskip

We observe from (\ref{E:lame-orth}) and (\ref{E:simple-orth-2}) that although the corresponding Lam\'e polynomial solutions and periodic BHE polynomial solutions of same degree respectively are orthogonal with respect to the parameters $\mu,\, \nu$, it is not clear if polynomials of \textit{different degrees} are orthogonal to each other. It turns out that an orthogonality exists for \textit{polynomials} of different degrees when they are formed from products of two Lam\'e polynomials of same degree but in different parameters. The following \textit{double orthogonality} can be found in \cite[pp. 379--381]{Heine}, \cite[pp. 467--469]{Hobson}, \cite[\S9.4]{Arscott1964}:
    \begin{equation}
        \label{E:lame-double-orth}
            \int_{-2K}^{2K}\int_{K-2iK^\prime }^{K+2iK^\prime} E_m^\mu(z)\, E_m^\mu(s)\,E_n^\nu(z)\, E_n^\nu(s)\,
            (\textrm{sn\,}^2z-\textrm{sn\,}^2s)\,dz\, ds=0
    \end{equation}
\medskip
whenever $n\not= m$, while if $m=n$, then the orthogonality still hold when $\mu\not=\nu$, for
$0\le\mu,\, \nu\le n$. On the other hand,
	\[
		 \int_{-2K}^{2K}\int_{K-2iK^\prime }^{K+2iK^\prime} [E_m^\mu(z)\, E_m^\mu(s)]^2\,
            (\textrm{sn\,}^2z-\textrm{sn\,}^2s)\,dz\, ds\not =0.
    \]
\medskip

 We show below that suitable product of two non-oscillatory/Liouvillian solutions of equation \eqref{E:BHE-periodic} also possesses a double orthogonality, over $\mathbf{R}\times(\mathbf{R}+i\pi)$, that is analogous to that of the Lam\'e equation discussed above, and that appears to be new.
\medskip

\begin{theorem}
    \label{T:double-orth}
     Let $K_3$ and $K_2$ be given real numbers. Let $m$ and $n$ be non-negative integers such that {$(K_0)_n<0$ and $(K_0)_m<0$} from Theorem \ref{T:determinant_soln} satisfy
		\begin{equation}
        \label{E:condition_II}
            \frac{K_3^2}{4}+K_2-2\sqrt{-(K_0)_n}=2(n+1),
    \end{equation}
 and
     \begin{equation}
        \label{E:condition_IIII}
            \frac{K_3^2}{4}+K_2-2\sqrt{-(K_0)_m}=2(m+1), 
    \end{equation}
respectively. Suppose $\mathrm{BH}_{n,\, \mu}\ (0\le \mu\le n)$ are solutions of the differential equation \eqref{E:BHE-eigenstate-II} as defined in the Theorem \ref{T:simple-orth} with finite exponent of convergence of zeros, and  $\mathrm{BH}_{m,\,\nu}\ (0\le \nu\le m)$ are the corresponding solutions to the equation \eqref{E:BHE-eigenstate-II} with $n$ replaced by $m$. Then we have
        \begin{equation}
            \label{E:double-orth}
             \int_{-\infty}^{\infty}\int_{-\infty+i\pi}^{\infty+i\pi}\, \mathrm{BH}_{n,\, \mu}(z)\, \mathrm{BH}_{n,\, \mu}(s)\, \mathrm{BH}_{m,\, \nu}(z)\,\mathrm{BH}_{m,\, \nu}(s)\,
              (e^z-e^s)\, dz\,ds=0
        \end{equation}
whenever $(n,\, \mu)\not=(m,\, \nu)$. Moreover,
		\begin{equation}
			\label{E:double-norm}
			\int_{-\infty}^{\infty}\int_{-\infty+i\pi}^{\infty+i\pi} [\mathrm{BH}_{n,\, \mu}(z)\, \mathrm{BH}_{n,\, \mu}(s)]^2\,
              (e^z-e^s)\, dz\,ds\not=0,
        	\end{equation}
if $(n,\, \mu)=(m,\, \nu)$.
\end{theorem}
\medskip

\section{Proof of Theorem \ref{T:simple-orth}} \label{S:simple-orth}
\begin{proof}
We note that since we assume \eqref{E:condition_III} and
 $\sqrt{-K_0}>0$ (so that $1+2\sqrt{-K_0}>0$), so the \eqref {E:condition_I} and the inequality \eqref{E:distinct} are fulfilled respectively. Theorem \ref{T:determinant_soln} guarantees that there are $n+1$ distinct solutions to \eqref{E:BHE-periodic} of the form  \eqref{E:BHE-periodic-soln},
where we recall that $Y_{n,\, \nu}$ is a polynomial. Hence we see that $\mathrm{BH}_{n,\,\nu}(z)\to 0$ as $z=x\to+\infty$. To see that $\mathrm{BH}_{n,\,\nu}(z)\to 0$ as $z=x\to-\infty$ we only need to note that $\sqrt{-K_0}>0$ because according to the assumption that $K_0$ is real and $K_0<0$. Moreover, the derivative assumes the form
	\begin{equation}
        \label{E:BHE-periodic-soln-3}
        		\begin{split}
            \mathrm{BH}^\prime_{n,\,\nu}(z)= \Big[-e^{-z}Y^\prime_{n,\,\nu}(e^{-z})&+\Big(\Big(\frac{K_3}{2}\,e^z-e^{2z}\Big)+(n+ \sqrt{-K_0})\Big)Y_{n,\,\nu}(e^{-z})\Big]\\
            &\times \exp\bigg[\Big(\frac{K_3}{2}\,e^z-\frac{1}{2}e^{2z}\Big)+(n+\sqrt{-K_0})\,z\bigg],
            	\end{split}
    \end{equation}
which is similar to that of $\mathrm{BH}_{n,\,\nu}(z)$. A similar analysis as that of  $\mathrm{BH}_{n,\,\nu}(z)$ reveals that $\mathrm{BH}^\prime_{n,\,\nu}(z)\to 0$ as $z=x\to\pm\infty$.

 Let $n$, $m$ be non-negative integers and that the coefficients $K_3,\, K_2$ and $K_0<0$  satisfy the relation (\ref{E:condition_I}) and \eqref{E:distinct}. Then it follows from Theorem \ref{T:determinant_soln}  that the equation \eqref{E:BHE-periodic} possesses $n+1$ distinct solutions $\mathrm{BH_{n,\, \nu}}(z)$ $(0\le \nu\le n)$  with finite exponent of convergence of zeros given by \eqref{E:BHE-periodic-soln}.

 Let $f_{n,\, \nu}=\mathrm{BH}_{n\, \nu}$ and  $f_{n,\, \mu}=\mathrm{BH}_{n,\, \mu}$ ($0\le \nu,\, \mu\le n$) be two solutions to (\ref{E:BHE-periodic}) with finite exponent of convergence. That is, $f_{n,\, \nu}$ and  $f_{n,\, \mu}$  are solutions to
  	\begin{equation}
		\label{E:before-substraction}
            \begin{split}
                    f_{n,\, \mu}^{\prime\prime}(z) &+\left(-e^{4z}+K_{3}\, e^{3z}+K_{2}\,e^{2z}+\big(K_{1}\big)_{n,\, \mu}\, e^{z}+K_0\right)\,f_{n,\, \mu}(z)=0\\
                    f_{n,\, \nu}^{\prime\prime}(z) &+\left(-e^{4z}+K_{3}\, e^{3z}+K_{2}\,e^{2z}+\big(K_{1}\big)_{n,\, \nu}\, e^{z}+K_0\right)\,f_{n,\, \nu}(z)=0
            \end{split}
        \end{equation}
respectively.
\medskip

We subtract the two equations resulting from multiplying the first equation of \eqref{E:before-substraction} throughout by $f_{n,\, \nu}$ and the second equation of \eqref{E:before-substraction} throughout by $f_{n,\, \mu}$ respectively. Then we integrate the resulting equation over $(-\infty,\, \infty)$. This yields
        \begin{equation*}
            \begin{split}
            &\int_{-\infty}^{\infty} \big[f_{n,\, \nu}(z)f_{n,\, \mu}^{\prime\prime}(z)-f_{n,\, \nu}(z)f_{n,\, \mu}^{\prime\prime}(z)\big]\, dz\\
            &\qquad +\big[\big(K_{1}\big)_{n,\, \mu}-\big(K_{1}\big)_{n,\, \nu}\big]\,\int_{-\infty}^{\infty}   f_{n,\, \nu}\,(z)f_{n,\, \mu}(z) \,e^z\, dz=0.
            \end{split}
        \end{equation*}
We deduce from the explicit representations \eqref{E:BHE-periodic-soln-2} of $\mathrm{BH}_{n,\, \nu}(z)$ and \eqref{E:BHE-periodic-soln-3} of $\mathrm{BH}^\prime_{n,\, \nu}(z)$ that they both vanish at $\pm\infty$. So integration-by-parts of the above equation yields
        \begin{equation*}
            \begin{split}
             &\int_{-\infty}^{\infty}  f_{n,\, \nu}(z)f_{n,\, \mu}(z) \,e^z\, dz \\
            &={\big[\big(K_{1}\big)_{n,\, \mu}-\big(K_{1}\big)_{n,\, \nu}\big]^{-1}} \int_{-\infty}^{\infty}   \big[f_{n,\, \nu}(z)f_{n,\, \mu}^{\prime\prime}(z)-f_{n,\, \nu}(z)f_{n,\, \mu}^{\prime\prime}(z)\big]\, dz\\
             &={\big[\big(K_{1}\big)_{n,\, \mu}-\big(K_{1}\big)_{n,\, \nu}\big]^{-1}} \big[f_{n,\, \nu}(z)f_{n,\, \mu}^{\prime}(z)-f_{n,\, \nu}(z)\,f_{n,\, \mu}^{\prime}(z)\big]_{-\infty}^{\infty} \\
            &=0
            \end{split}
        \end{equation*}
thus proving that the two functions $f_{n,\, \nu}(z)$, $f_{n,\, \mu}(z)$ are orthogonal, since the corresponding characteristic values $\big(K_{1}\big)_{n,\, \mu},\, \big(K_{1}\big)_{n,\, \nu}$ are distinct. That is, whenever $\mu\not=\nu$ for $0\le \mu,\, \nu\le n$. This proves \eqref{E:simple-orth-2}.
\smallskip

Let us now consider the case when $\mu=\nu$.  Notice that the product between $(\mathrm{BH}_{n,\, \mu}(z))^2$ and the weight function $ e^z$ yields, under the assumption on $2 \sqrt{-K_0}+1>0$, that
		\begin{equation}
		\begin{split}
		(\mathrm{BH}_{n,\, \mu}(z))^2 \,e^z &=
		\big[Y_{n,\, \nu}(e^{-z})\big]^2 \exp\big[(K_3e^z-e^{2z})+2(n+ \sqrt{-K_0})z\big]\, e^z\\
		&=\big[Y_{n,\, \nu}(e^{-z})\big]^2 \exp\big[(K_3e^z-e^{2z})+(2n+1+2 \sqrt{-K_0})z\big]>0
		\end{split}
	\end{equation}
holds throughout the real-axis $\mathbf{R}$.  But the representation \eqref{E:BHE-periodic-soln-2} clearly shows that the product
	\[
		(\mathrm{BH}_{n,\, \mu}(z))^2\, e^z=\big[Y_{n,\, \nu}(e^{-z})\big]^2 \exp\big[(K_3e^z-e^{2z})+(2n+1+2 \sqrt{-K_0})z\big]
	\]
vanishes at $\pm\infty$ sufficiently fast  to guarantee that the integral
	\begin{equation*}
			 \int_{-\infty}^{\infty} \big(\mathrm{BH}_{n,\, \mu}(x)\big)^2 \,e^x\, dx
    \end{equation*}
converges and non-vanishing. This proves  \eqref{E:single-norm}.
\end{proof}
\medskip

 \section{Proof of Theorem \ref{T:double-orth}} \label{S:double-orth}
\begin{proof} The idea of the proof is classical by constructing a suitable pair of partial differential equations and applying integration-by-parts \cite[\S 276]{Hobson} (see also \cite[pp. 379--381]{Heine}.
Let $K_3,\ K_2$ be real, and $(K_0)_n,\ (K_0)_m$ satisfy the equations \eqref{E:condition_II} and \eqref{E:condition_IIII} respectively. Let $f_{n,\, \mu}(z)=\mathrm{BH}_{n,\, \nu}(z)$ and $f_{m,\, \nu}(z)=\mathrm{BH}_{m,\, \mu}(z)$ be two corresponding eigen-solutions of the \eqref{E:BHE-periodic} respectively. We define
       \[
            F_{n,\, \mu}(z,\, s):=f_{n,\, \mu}(z) f_{n,\, \mu}(s),\qquad
            F_{m,\, \nu}(z,\, s):=f_{m,\, \nu}(z) f_{m,\, \nu}(s)
        \]
to be complex functions of two variables $(z,\, s)$. Clearly they satisfy, respectively, the following partial differential equations
        \begin{align}
                \frac{\partial^2 F_{n,\, \mu}}{\partial z^2}-\frac{\partial^2 F_{n,\, \mu}}{\partial s^2}
                    &=\Big[-(e^{4z}-e^{4s})+ K_3\,(e^{3z}-e^{3s})\notag\\
                    &\qquad+K_{2}\,(e^{2z}-e^{2s})+\big(K_{1}\big)_{n,\, \mu}\, (e^{z}-e^s)\Big] \, F_{n,\, \mu}=0\label{E:PDE_I}\\
                    \frac{\partial^2 F_{m,\, \nu}}{\partial z^2}-\frac{\partial^2 F_{m,\, \nu}}{\partial s^2}
                    &=\Big[-(e^{4z}-e^{4s})+ K_3\,(e^{3z}-e^{3s})\notag\\
                    &\qquad +K_{2}\,(e^{2z}-e^{2s})+\big(K_{1}\big)_{m,\, \nu}\, (e^{z}-e^s)\Big] \, F_{m,\, \nu}=0\label{E:PDE_II}
            \end{align}

Subtracting the equations (\ref{E:PDE_I}) after multiplying  by $F_{m,\, \nu}$ from the equation (\ref{E:PDE_II}) after multiplying by $F_{n,\, \mu}$. We integrate the resulting difference with respect to the two variables $(z,\, x)$ from $-\infty$ to $\infty$ and $-\infty+\pi i$ to $\infty+\pi i$ respectively. This yields
	   \begin{equation*}
                \begin{split}
                     &\int_{-\infty}^{+\infty}\int_{-\infty+\textrm{i}\pi}^{+\infty+\textrm{i}\pi}
                     \Big\{
                        F_{m,\, \nu}\big[(F_{n,\, \mu}\big)_{zz}-(F_{n,\, \mu})_{ss}\big]
                        -
                        F_{n,\, \mu}\big[(F_{m,\, \nu}\big)_{zz}-(F_{m,\, \nu})_{ss}\big]\Big\}\, dz\, ds\\
                        &+
                        \big[(K_1)_{n,\,\mu}-(K_1)_{m,\,\nu}\big]\int_{-\infty}^{+\infty}\int_{-\infty+\textrm{i}\pi}^{+\infty+\textrm{i}\pi}
                        F_{n,\, \mu}(z,\, s)\,F_{m,\, \nu}(z,\, s) \, (e^{z}-e^s)\, dz\, ds=0.
                \end{split}
            \end{equation*}

We note that
 $(F_{m,\, \nu})_z=  \mathrm{BH}^\prime_{m,\, \nu}(z)\cdot \mathrm{BH}_{n,\, \nu}(s) $ and    $(F_{n,\, \nu})_s=  \mathrm{BH}_{n,\, \nu}(z) \cdot \mathrm{BH}^\prime_{n,\, \nu}(s) $. It again follows from  \eqref{E:BHE-periodic-soln-2} $\mathrm{BH}_{n,\, \nu}(z)$ and  \eqref{E:BHE-periodic-soln-3} $\mathrm{BH}^\prime_{n,\, \nu}(z)$ vanish at $\pm\infty$, so that both $(F_{m,\, \nu})$ and $(F_{m,\, \nu})_z$ also vanish at $\pm\infty$, and the same holds partial derivatives with $z$ replaced by $s$.

Integration of the above equation by parts with respect to both variables $(z,\, s)$ and utilising our observation in the last paragraph that $(F_{m,\,  \nu})$ and $(F_{m,\, \nu})_z$ also vanish simultaneously at $\pm\infty$ yields
		\begin{align*}
                       0&=\int_{-\infty}^{\infty}\int_{-\infty+\textrm{i}\pi}^{\infty+\textrm{i}\pi}\big[F_{m,\, \nu}\,(F_{n,\, \mu}\big)_{zz}-
                        F_{n,\, \mu}\,(F_{m,\, \nu})_{zz}\big]\\
                 &\quad+\big[F_{n,\, \mu}\, (F_{m,\, \nu})_{ss}-F_{m,\, \nu}\,(F_{n,\, \mu}\big)_{ss}\big]\,
                        dz\,ds\\
                &\quad+\big[(K_1)_{n,\,\nu}-(K_1)_{m,\,\mu}\big]\int_{-\infty}^{\infty}\int_{-\infty+\textrm{i}\pi}^{\infty+\textrm{i}\pi}
                        F_{n,\, \mu}(z,\, s)\,F_{m,\, \nu}(z,\, s) \, (e^{z}-e^s)\, dz\, ds\\             &=  \int_{-\infty+\textrm{i}\pi}^{\infty+\textrm{i}\pi}\big[F_{m,\, \nu}\,(F_{n,\, \mu}\big)_{z}-
                        F_{n,\, \mu}\,(F_{m,\, \nu})_{z}\big]\big|_{\infty}^{-\infty}\,ds\\
                &\quad -\int_{-\infty}^{\infty}\int_{-\infty+\textrm{i}\pi}^{\infty+\textrm{i}\pi} (F_{m,\, \nu})_z\,(F_{n,\, \mu})_{z}-
                        (F_{n,\, \mu})_z\,(F_{m,\, \nu})_{z}\, dz\, ds\\
                &\quad +
                  \int_\infty^{-\infty}\big[F_{n,\, \mu}\, (F_{m,\, \nu})_s-F_{m,\, \nu}\,(F_{n,\, \mu})_s\big]\big|_{-\infty+\textrm{i}\pi}^{\infty+\textrm{i}\pi}\, dz
                  \end{align*}
              \begin{align*}
                & \quad -\int_{-\infty}^{\infty}\int_{-\infty+\textrm{i}\pi}^{\infty+\textrm{i}\pi} (F_{m,\, \nu})_s\,(F_{n,\, \mu})_s-
                        (F_{n,\, \mu})_s\,(F_{m,\, \nu})_s\, dz\, ds\\
                &\quad +\big[(K_1)_{n,\,\nu}-(K_1)_{m,\,\mu}\big]\int_{-\infty+\textrm{i}\pi}^{\infty+\textrm{i}\pi}\int_{\infty}^{-\infty}
                        F_{n,\, \mu}(z,\, s)\,F_{m,\, \nu}(z,\, s) \, (e^{z}-e^s)\, dz\, ds
            \end{align*}
            \begin{align}\begin{split}
                &=\int_{-\infty+\textrm{i}\pi}^{\infty+\textrm{i}\pi}\big[F_{m,\, \nu}\,(F_{n,\, \mu}\big)_{z}-
                        F_{n,\, \mu}\,(F_{m,\, \nu})_{z}\big]\big|_{-\infty }^{\infty}\,ds-0\\
                & \quad +
                \int_{-\infty}^{\infty} \big[F_{n,\, \mu}\, (F_{m,\, \nu})_{s}-F_{m,\, \nu}\,(F_{n,\, \mu})_{s}\big]\big|_{-\infty+\textrm{i}\pi}^{\infty+\textrm{i}\pi}\, dz-0\\
                 &\quad+\big[(K_1)_{n,\,\nu}-(K_1)_{m,\,\mu}\big]\int_{-\infty}^{\infty}\int_{-\infty+\textrm{i}\pi}^{\infty+\textrm{i}\pi}
                        F_{n,\, \mu}(z,\, s)\,F_{m,\, \nu}(z,\, s) \, (e^{z}-e^s)\, dz\, ds
            \end{split}
        \end{align}
where the integrands in the second and the fourth double integrals are identically zero. Moreover, both of the two remaining single integrals in the last equal sign are also identically zero since the  $F_{n,\, \mu}$, $F_{m,\, \nu}$ and their partial derivatives vanish simultaneously at $\pm\infty$. That is, we have shown that
	\[
		\big[(K_1)_{n,\,\nu}-(K_1)_{m,\,\mu}\big]\int_{-\infty}^{\infty}\int_{-\infty+\textrm{i}\pi}^{\infty+\textrm{i}\pi}
                        F_{n,\, \mu}(z,\, s)\,F_{m,\, \nu}(z,\, s) \, (e^{z}-e^s)\, dz\, ds=0,
    \]
thus proving the (\ref{E:double-orth}) holds whenever $n\not= m$ and irrespective to the choices of $\nu$ and $\mu$.

It remains to consider the case $m=n$ in (\ref{E:double-orth}). We can easily rewrite
		\[
			\int_{-\infty}^{\infty}\int_{-\infty+\textrm{i}\pi}^{\infty+\textrm{i}\pi}\Big(\mathrm{BH}_{n,\, \nu}(s)\, \mathrm{BH}_{n,\nu}(z)\Big)^2(e^z-e^s)\, dzds
		\]
into the following double integral
	\begin{equation*}
		\begin{split}
			&\int_{-\infty}^\infty\int_{-\infty+\textrm{i}\pi}^{\infty+\textrm{i}\pi}
			\mathrm{BH}^2_{n,\, \nu}(z)\cdot \mathrm{BH}^2_{n,\, \nu}(s) (e^z-e^s)\, dz\, ds\\
			&=\int_{-\infty}^\infty\int_{-\infty}^{\infty}
			[\mathrm{BH}^2_{n,\, \nu}(z)\cdot \mathrm{BH}^2_{n,\, \nu}(\xi+\mathrm{i}\pi)\, (e^z+e^\xi)]\, dz\, d\xi.		
			\end{split}
	\end{equation*}
We note that it follows from \eqref{E:BHE-periodic-soln-2} that the $\mathrm{BH}^2_{n,\, \nu}(\xi+\mathrm{i}\pi)$ converges to $0$ rapidly when $\xi\to\pm\infty$, and that the weight $(e^z+e^\xi)$ is positive throughout the integration region. Hence we deduce that the double integral above, in much the same way as we have done for the validity of \eqref{E:single-norm}, converges and non-vanishing. This proves the \eqref{E:double-norm}.
\end{proof}

\section{Fredholm integral equations}
\label{S:integral}
Whittaker commented in \cite[p. 15]{Whitt1914} that unlike the hypergeometric equation where its solutions have integral representations, certain solutions of Heun equations satisfy \textit{homogeneous Fredholm-type integral equations of second kind} of the form
    \begin{equation}
        \label{E:real-integral}
            u(x)=\lambda\int_a^b K(x,\, t)\, u(t)\, dt,
    \end{equation}
 instead (see, e.g., \cite[Part A, \S6]{Ronv1995}), where $\lambda$ is the eigenvalue of the solution $u(x)$, and {the kernel $K(x,\, t)$ is symmetric in $x$ and $t$}, if any. Such integral equations are of fundamental importance for Heun equations (see, for example, \cite[Part A]{Ronv1995}), {which play the role of integral representations for hypergeometric equation.}

Whittaker \cite{Whitt1915} appears to be the first one who showed that the Lam\'e polynomials, being eigen-solutions to the Lam\'e equation (\ref{E:lame-1}), satisfy the following \textit{Fredholm integral equation of the second kind}:
        \begin{equation}
            \label{E:lame-integral}
                y(z)=\lambda\,\int_{-2K}^{2K} P_n(k\, \textrm{sn\,}z\,
                \textrm{sn\,} t)\, y(t)\, dt,
        \end{equation}
where the $P_n(t)$ is the $n-$th Legendre polynomial and $\lambda$ the corresponding eigen-value. We refer the reader about further improvements to Whittaker's result to consult \cite{Ronv1995}.

\medskip

It has long been known that \textit{Laplace's} and analogous transforms method can solve {some} linear differential equations by definite integrals. The kernels of these integral equations are constructed by solving a specifically designed linear \textit{partial differential equations}, which are based on the adjoint form of the original differential equation, by the method of \textit{separation of variables}.
See for example,
Ince \cite[XVIII]{Ince1926}. Such \textit{kernel functions} approach\footnote{See Garabedian \cite[Chapter 10]{Garabedian1964} for a modern treatment.}
, in the words of Ince \cite{Ince1922}, differs from Green's function consideration where the kernel of the integrals involved usually has discontinuities along the ``diagonal", while the Laplace-type integral transform method produces ``continuous" kernels. Our principal concern here is to obtain such an explicit Fredholm-type integral equation and solutions for the PBHE (\ref{E:BHE-periodic}), using the Laplace method.

 \begin{theorem}\label{T:periodic_integral_Eqn} Let $K_3,\, K_2$ and $K_0$ be real such that
       \begin{equation}
                \label{E:distinct-2}
                   {K_0}=-m^2,\qquad m\in \mathbf{N}.
        \end{equation}
For each non-negative integer $n$, there are $n+1$ distinct pairs of \textit{generalised eigenvalues} $\big((K_1)_{n,\, \nu},\,\lambda_{n,\, \nu}\big),\ 0\le\nu\le n$ such that the Fredholm integral equation of the second kind
    \begin{align}
        \label{E:periodic-integral-eqn}
             y(z)&=\lambda \int_{0}^{4\pi i}
e^{\frac{\epsilon_\infty}{2}\big(K_3(e^z+e^t)-(e^{2z}+e^{2t})\big)+\epsilon_0\sqrt{-K_0}\,(z+t)}
\\
	&\qquad\qquad\times \Phi\Big(\frac{\epsilon_\infty K_3^2}{16}+\frac{1+\epsilon_0\sqrt{-K_0}}{2}+\frac{\epsilon_\infty K_2}{4};\,\,\frac{1}{2};\,
            -\epsilon_\infty(e^z+e^s-\frac{K_3}{2})^2 \Big)e^{t}\, y(t)\, dt,
    \end{align}
where
$\Phi(a,c;x)$ is the Kummer function, admits corresponding eigen-solutions
defined by \emph{(\ref{E:BHE-periodic-soln})} and  $(K_1)_{n,\, \nu}$ satisfies the determinant
        \[
            \det\big( (K_1)_{n,\, \nu}\big)=0
        \]
given in {(\ref{E:Det})}.
\end{theorem}

\begin{proof} Without loss of generality, we may consider for each integer $n\in\mathbf{N}$, and $\nu=0,\, 1,\, \cdots, n$, the coefficients $K_3,\, K_2,\, K_0$ satisfy the relation (\ref{E:condition_I}). Then there are $n+1$ of $K_1=(K_1)_{n,\, \nu}$ that are roots to the determinant $D_{n+1}(K_1)=0$ from (\ref{E:Det}),
the equation
\begin{equation}
    \label{E:BHE-periodic_rotated}
        e^{-z}f^{\prime\prime}(z)+\left(K_{4}\,e^{3z}+K_{3}\, e^{2z}+K_{2}\,e^{z}+K_{1} +K_0\, e^{-z}\right)f(z)=0,
    \end{equation}
with $K_4=-1$.
\bigskip

Suppose (\ref{E:BHE-periodic_rotated}) admits an ``eigen-solution" $u(z)$. Then we define a sequence of second order partial differential operators
    \begin{equation}
		\label{E:schrodinger}
        L_z:=e^{-z}\Big(\frac{\partial^2}{\partial z^2}+\ell(z)\Big),
    \end{equation}
where {$L_z=(L_{n,\, \nu})_z$ for each integer $n\ge 0$ and}
    \begin{equation}
        \ell(z):={(\ell_{n,\, \nu})(z)=}K_{4}\,e^{4z}+K_{3}\, e^{4z}+K_{2}\,e^{2z}+K_{1}e^z+K_0.
    \end{equation}
Let $K(z,\, s)$ be a function with two complex variables. Then we construct a partial differential equation for which $K(z,\, s)$ solves:
    \begin{equation}
        \label{E:PDE}
            L_z(K)-L_s(K)=e^{-z}\frac{\partial^2K}{\partial z^2}-e^{-s}\frac{\partial^2 K}{\partial s^2}+[e^{-z}\ell(z)-e^{-s}\ell(s)]\, K.
    \end{equation}
Now we put
    \begin{equation}
        \label{E:kernel}
        K(z,\, s)=\phi(z)\,\phi(s)\, F(\zeta),\quad \zeta=e^z+e^s,
    \end{equation}
with $\phi(z)=\exp\big[a\, e^{2z}+b\, e^z +dz\big]$. Here {$a,\, b,\,d$} are constants {remain to be chosen}.
\medskip

Substituting (\ref{E:kernel}) into (\ref{E:PDE}) yields
    \[
        L_z(K)-L_s(K)=e^{-z}K_{zz}-e^{-s}K_{ss}+\big[e^{-z}\ell(z)-e^{-s}\ell(s)\big]\cdot K
    \]
that is,
\begin{align}
  L_z(K)-L_s(K) &= e^{-z}\frac{\partial^2 K(z,s)}{\partial z^2}-e^{-s}\frac{\partial^2 K(z,s)}{\partial s^2}\\
	&= \Big[\Big(\frac{\phi''(z)}{\phi(z)}e^{-z}-\frac{\phi''(s)}{\phi(s)}e^{-s} \Big)+2\frac{F'}{F}\Big(\frac{\phi'(z)}{\phi(z)}-\frac{\phi'(s)}{\phi(s)} \Big)\notag\\
&\quad+\frac{F''}{F}(e^z-e^s)+\sum_{j=0}^4K_j(e^{(j-1)z}-e^{(j-1)s})\Big]K(z,\,s),
\end{align}
which vanishes identically if we set
    \begin{equation}
        \label{E:two-set-parameters-I}
        {4a^2+K_4=0},\quad d^2+K_0=0,
    \end{equation}
    and
    \begin{align}\label{E:Kummer}
    F''(\zeta)+(4a\zeta+2b)F'(\zeta)+\Big((K_3+4ab)\zeta+K_2+4a(1+d)+b^2\Big)F(\zeta)=0.
    \end{align}
Note that $K_4=-1$. We take $a={\epsilon_\infty}/{2},\, b=-{\epsilon_\infty K_3}/{2},\, d=\epsilon_0\sqrt{-K_0}$, ($\epsilon_0^2=\epsilon_\infty^2=1$). Let $z=-\epsilon_\infty(\zeta-K_3/2)^2$, $F(\zeta)=u(z)$. Then we can transform the differential equation \eqref{E:Kummer} to
		\begin{equation}
			\label{E:1F1}
			zu^{\prime\prime}+(1/2-z)u^\prime(z)-\Big(\epsilon_\infty K_3^2/16+\epsilon_0\sqrt{-K_0}/2+\epsilon_\infty K_2/4+1/2\Big)u=0,
		\end{equation}which is the standard form of the confluent hypergeometric equation \cite[Chap. 6]{Bateman_I}. It is well-known that the Kummer function $\Phi(a;\, c; z)$, where $a=\epsilon_\infty K_3^2/16+\epsilon_0\sqrt{-K_0}/2+\epsilon_\infty K_2/4+1/2$ and $c=1/2$, is an entire solution to the confluent hypergeometric equation \cite[\S6.1 (2)]{Bateman_I}.
	
Hence
	 \begin{align}
        \label{E:K}
            F(\zeta)&=\Phi\Big(\frac{\epsilon_\infty K_3^2}{16}+\frac{1+\epsilon_0\sqrt{-K_0}}{2}+\frac{\epsilon_\infty K_2}{4};\,\frac{1}{2};\,
            \frac{-\epsilon_\infty(2\zeta-K_3)^2}{4} \Big)\notag \\
			&=\Phi\Big(\frac{\epsilon_\infty K_3^2}{16}+\frac{1+\epsilon_0\sqrt{-K_0}}{2}+\frac{\epsilon_\infty K_2}{4};\,\,\frac{1}{2};\,
            -\epsilon_\infty(e^z+e^s-\frac{K_3}{2})^2 \Big).
        \end{align}
Note that the $F(\zeta)$ above is a periodic function of period $2\pi i$. Consider the operator
 \begin{align}\label{E:operator}
      R_z\,w & :=\frac{d}{dz}\Big(e^{-z}\frac{d}{dz}w(z)\Big)\notag\\
			&\qquad +\Big(K_{4}\,e^{3z}+K_{3}\, e^{2z}+K_{2}\,e^{z}+K_{1}+(K_0+\frac{1}{4})\,e^{-z}\Big)w(z)=0
      \end{align}
which is self-adjoint $R_z=R_z^*$, and where
	\begin{equation}
		\label{E:gauge}
			R_z:=e^{z/2}L_ze^{-z/2}
	\end{equation}
is a gauge transform of $L_z$ defined in \eqref{E:schrodinger}.  We define
	\begin{equation}\label{E:T}
		T(z):=\int_{\Gamma} e^{(z+t)/2}\, K(z,\,t)\, w(t)\, dt
	\end{equation}
where $w(t)$ is an eigen-solution to \eqref{E:BHE-operator} of the form \eqref{E:BHE-periodic-soln} and $\Gamma$ denote the line segment $[0,\, 4\pi i]$. Applying the operator $R_z$ to $T(z)$ and applying the gauge transform yields,
	\begin{align}
		\label{E:K-T}
		R_z T(z) &=\int_\Gamma R_z[e^{z/2}K(z,\,t)] e^{t/2}w(t)\, dt=\int_{\Gamma} e^{z/2}L_z[K(z,\,t)]\, e^{t/2}w(t)\, dt\notag\\
			&=e^{z/2}\int_{\Gamma} e^{t/2}\, L_t[e^{-t/2}\cdot e^{t/2} K(z,\,t)]\,w(t)\, dt\notag\\
			&=e^{z/2}\int_{\Gamma} R_t[e^{t/2}K(z,\,t)]\,w(t)\, dt,
	\end{align}
where the $K_t$ assumes the same form as \eqref{E:operator} with $z$ replaced by $t$. But then integration-by-parts twice yields
	\begin{align}
		&\int_{\Gamma} R_t[e^{t/2}  K(z,\,t)]\,w(t)\, dt\notag\\
		&=\int_{\Gamma}\Big\{ \frac{d}{dt}\Big[e^{-{t}}\frac{d}{dt}\big( e^{t/2}K(z,\,t)\big)
\Big]w(t)+\Big[\sum_{j=1}^4 K_{j}\,e^{jt}+\big(K_0+\frac{1}{4}\big)\Big]
\big( e^{t/2} K(z,\,t)\big)w(t)\Big\}\, dt  \notag\\
		&=e^{-{t}}\frac{d}{dt}\big( e^{t/2} K(z,\,t)\big)\,w(t)\Big|_\Gamma
		-\int_{\Gamma} \Big[e^{-{t}}\frac{d}{dt}\big( e^{t/2} K(z,\,t)\big)\,w(t)\Big]
		w^\prime(t)\, dt\notag\\
		&\hskip5cm +\int_{\Gamma} \Big[\sum_{j=1}^4 K_{j}\,e^{jt}+\big(K_0+\frac{1}{4}\big)\Big]
\big( e^{t/2} K(z,\,t)\big)w(t)\, dt  \notag\\
		&=0-\big( e^{t/2} K(z,\,t)w(t)\big)e^{-t}w^\prime(t)\Big|_\Gamma
		+\int_{\Gamma} \big( e^{t/2} K(z,\,s)w(t)\big)\big(e^{-t}w^\prime(t)\big)^\prime\, dt\notag \\
		&\hskip5cm +\int_{\Gamma} \Big[\sum_{j=1}^4 K_{j}\,e^{jt}+\big(K_0+\frac{1}{4}\big)\Big]
\big( e^{t/2} K(z,\,t)\big)w(t)\, dt  \notag\\
		&=0+\int_{\Gamma} \big( e^{t/2} K(z,\,t)\,w(t)\big)\cdot R_t[w(t)]\, dt=0,
	\end{align}
since both the $e^{t/2} K(z,\,t)$, $w(t)$ and $w^\prime(t)$ return to the same values after a $4\pi i$ shift and that $R_t[w(t)]\equiv 0$.  Combining this with \eqref{E:T} shows that $T(z)$ is a solution to the equation \eqref{E:K-T}. We have thus proved, apart from a non-zero constant, that the right-hand side of the \eqref{E:periodic-integral-eqn}, which we denote by $\tilde{T}(z)$ is a solution to \eqref{E:BHE-periodic_rotated}. Since $\sqrt{-K_0}\in \mathbf{Z}$ so the $\tilde{T}(z)$  and moreover even its derivative $\tilde{T}^\prime(z)$ are both periodic of period $4\pi i$. It follows from standard Sturm-Liouville theory (see e.g.  \cite[\S2.2, (2.2.1)]{Eastham1973}) asserts that there is sequence of real eigen-values $\lambda_{n,\, \nu}\ (\nu=0,\,1,\, 2,\, \cdots)$ and $\lambda_{n,\, \nu}\to\infty$ as $\nu\to \infty$. It is known that corresponding to each $\lambda_{n\,\nu}$ there can be at most one such eigenfunction that satisfies the boundary condition. Since both the $\mathrm{BH}_{n\,\nu}$ and $\tilde{T}$ satisfy the same equation \eqref{E:BHE-periodic_rotated} and same boundary condition $f(0)=f(4\pi i)$ and $f^\prime(0)=f^\prime(4\pi i)$, so $\mathrm{BH}_{n\,\nu}$ and $\tilde{T}$ can differ by at most a non-zero constant. This completes that proof. We note that for each non-negative integer $n$, only the first $n+1$ solutions of the form \eqref{E:BHE-periodic-soln} correspond to the eigen-values $\lambda_{n,\, \nu}\ (0\le \nu\le n)$ are eigensolutions to the integral equation \eqref{E:periodic-integral-eqn} here.
 \end{proof}

\section{Comments and conclusions}

We would like to point out that it is known that the Lam\'e equation is a periodic version of a limiting case of the \textit{Heun equation} \cite{Ronv1995} and our equation PBHE \eqref{E:BHE-periodic} is a periodic version of the \textit{biconfluent Heun equation} \eqref{E:BHE} \cite{Ronv1995}, which is sometime called \textit{rotating harmonic oscillator}  \cite{Masson1982}. Despite the long history of BHE (see e.g. \cite{schrodinger}, \cite{Dunham1932}), our understanding of the
equation is still far from satisfactory \cite[\S6]{Masson1982}. As far as the authors are aware, the PBHE first appears in Turbiner's study of \textit{quasi-exact solvable} differential operators related to Lie algebra $sl(2)$ consideration \cite[Eqn. VII, Table 1]{Turbiner1988}. This paper appears to be the first serious study of the PBHE from Hill's equation viewpoint using differential Galois theory and Nevanlinna's value distribution.   In particular, we have shown that the Liouvillian solutions of the PBHE are precisely those solutions which have finite exponent of convergence of zeros (non-oscillatory). We then show that these Liouvillian solutions exhibits novel orthogonality relations. 

We have just learnt that our orthogonality results for PBHE can be explained with a new \textit{jointly orthogonal polynomials} theory proposed recently by Felder and Willwacher \cite{FW2015} in the final stage of preparation of this paper. See also \cite{Arscott1972}. Their theory also covers orthogonality for the Lam\'e and Whittaker-Hill equations. However, our PBHE and orthogonality weight is more general than what is contained in \cite{FW2015} (i.e., $K_3=0$ in \eqref{E:BHE-periodic}). On the other hand, it is tempting to think that our orthogonality results for the PBHE are simple "pull-back" of those in \cite{FW2015}. We argue that there are major differences between those differential equations with rational potentials and their periodic counterparts. First, usually, both the orthogonality and integral equations results for the periodic equations are much more elegant than their "rational counterparts". Second, sometimes, certain results only exist for equations with periodic potentials. For example, although the \textit{double-confluent Heun equation} (DHE) \cite{Ronv1995} can only have asymptotic expansions for solutions at the origin $x=0$ which is an irregular singularity, Luo and the first author of this paper have obtained in \cite{Chiang-Luo2016} both the general and (anti-)periodic entire solutions for a periodic counterpart of the DHE, namely the Whittaker-Hill equation \eqref{E:WH-eqn}. An additional advantage is that periodic equation allows us to utilize the far-reaching classical Floquet theory \cite{Magnus1976}, \cite{Eastham1973}. In the case of Fredholm integral equation of the second kind with periodic symmetric kernel, the boundary condition is also much more simpler than their rational kernel counterparts \cite{Maroni1979}.

Finally we recall that Bank \cite{Bank1993} suggested an algorithm involving the construction of what he called \textit{approximate square-roots} to find explicit representations of non-oscillatory solutions to the \eqref{E:BL-general}. We would like to point out that Bank's algorithm is essentially a special case of the case (1) of Kovacic's algorithm that we have applied in this paper. We shall pursue this matter in a subsequent project. 

\section*{\ackname}
	The authors are deeply indebted for the referee's constructive comments to our paper, and for his/her generosity in sharing with us, and sometimes with detail explanation, of a number of useful references in differential Galois theory. The authors  would also like to acknowledge very useful discussions with their colleague Avery Ching that eventually led to this project. Finally, the authors thank the editors for their excellent stylistic suggestions.
	
 \section{Appendix: Kovacic's algorithm}
Let us fix the notation. For
$$ r=\frac{s}{t}, \qquad s,\, t \in \mathbf{C}[x],$$
(1) Denote by $\Gamma'$ be the set of (finite) poles of $r$, i.e., $\Gamma'=\{c\in \mathbf{C}: t(c)=0 \}$.\\
(2) Denote by $\Gamma=\Gamma'\cup\{\infty \}$\\
(3) By the order of $r$ at $c\in \Gamma', o(r_c)$, we mean the multiplicity of $c$ as a pole of $r$.\\
(4) By the order of $r$ at $\infty, o(r_{\infty})$, we mean the order of pole of $r(1/x)$ at $x=0$. That is $o(r_{\infty})=\deg(t)-\deg(s)$.

We list only the case 1 out of the four cases in the original Kovacic algorithm here. We refer to either to Kovacic's original article \cite{Kovacic1986}, or \cite{Acosta_Blazquez_2008} (see also \cite{Morales-Ruiz2015}) for the full algorithm.

The first case of four sub-cases.\\

In this case $[\sqrt{r}]_c$ and $[\sqrt{r}]_{\infty}$ means the Laurent series of $\sqrt{r}$ at $c$
and the Laurent series of $\sqrt{r}$ at $\infty$ respectively. Furthermore, we define $\epsilon(p)$ as follows: if
$p\in \Gamma$, then  $\epsilon(p)\in \{+,- \}$. Finally, the complex numbers $\alpha_c^+,\alpha_c^-,\alpha_{\infty}^+,\alpha_{\infty}^-$ will be defined in the first step. If the differential
equation  has no poles, then it can only fall in this case.

\textbf{Step 1.} Search for each $c\in \Gamma'$ and for $\infty$ the corresponding situation as follows:
\begin{itemize}
\item If $o(r_c)=0$, then
   \begin{align*}
    [\sqrt{r}]_c=0,\quad \alpha_c^{\pm}=0.
   \end{align*}
\item If $o(r_c)=1$, then
   \begin{align*}
    [\sqrt{r}]_c=0,\quad \alpha_c^{\pm}=1.
   \end{align*}
\item If $o(r_c)=2$, and
   \begin{align*}
    & r=\cdots+b(x-c)^{-2}+\cdots,\mbox{then}\\
      & [\sqrt{r}]_c=0,\quad \alpha_c^{\pm}=\frac{1\pm\sqrt{1+4b}}{2}.
   \end{align*}
 \item If $o(r_c)=2v\geq 4$, and
   \begin{align*}
    & r=(a(x-c)^{-v}+\cdots+d(x-c)^{-2})^2+b(x-c)^{-(v+1)}+\cdots,\mbox{then}\\
      & [\sqrt{r}]_c=a(x-c)^{-v}+\cdots+d(x-c)^{-2},\quad \alpha_c^{\pm}=\frac{1}{2}\Big(\pm\frac{b}{a}+v\Big).
   \end{align*}
\item If $o(r_{\infty})>2$, then
   \begin{align*}
    [\sqrt{r}]_{\infty}=0,\quad \alpha_{\infty}^{+}=0,\quad \alpha_{\infty}^{-}=1.
   \end{align*}
\item If $o(r_{\infty})=2$ and $r=\cdots+bx^{-2}+\cdots$, then
   \begin{align*}
    [\sqrt{r}]_{\infty}=0,\quad \alpha_{\infty}^{\pm}=\frac{1\pm\sqrt{1+4b}}{2}.
   \end{align*}
\item If $o(r_{\infty})=-2v\leq 0$, and
   \begin{align*}
   & r=(ax^v+\cdots+d)^2+bx^{v-1}+\cdots,\mbox{then}\\
   & [\sqrt{r}]_{\infty}=ax^v+\cdots+d ,\quad \alpha_{\infty}^{\pm}=\frac{1}{2}\Big(\pm\frac{b}{a}-v\Big).
   \end{align*}
\end{itemize}
\textbf{Step 2}. Find $D\neq \emptyset$ defined by
\begin{align*}
D=\{ m\in \mathbf{Z}_{+}: m=\alpha_{\infty}^{\epsilon_{\infty}}-\sum_{c\in \Gamma'}\alpha_c^{\epsilon_c},\, \forall(\epsilon_p)_{p\in \Gamma} \}.
\end{align*}
If $D=\emptyset$, then we should start with the case 2. Now if $\# D>0$, then for each $m\in D$ we search $\omega\in \mathbf{C}(x)$ such that
$$ \omega=\epsilon(\infty)[\sqrt{r}]_{\infty}+\sum_{c\in \Gamma'}(\epsilon(c)[\sqrt{r}]_c+\alpha_c^{\epsilon(c)}(x-c)^{-1}).$$
\textbf{Step 3}. For each $m\in D$, search for a monic polynomial $P_m$ of degree $m$ with
\[  P_m''+2\omega P_m'+(\omega'+\omega^2-r)P_m=0.\]
If one is successful, then $y_1=P_me^{\int \omega}$ is a solution of the differential equation. Else, Case 1 cannot hold.


\end{document}